\documentclass[11pt]{amsart}
\usepackage[T1]{fontenc}
\usepackage{lmodern}
\usepackage{amssymb,amsmath}
\usepackage{ifxetex,ifluatex}
\usepackage{fixltx2e} 
\IfFileExists{upquote.sty}{\usepackage{upquote}}{}
\ifnum 0\ifxetex 1\fi\ifluatex 1\fi=0 
  \usepackage[utf8]{inputenc}
\else 
  \ifxetex
    \usepackage{mathspec}
    \usepackage{xltxtra,xunicode}
  \else
    \usepackage{fontspec}
  \fi
  \defaultfontfeatures{Mapping=tex-text,Scale=MatchLowercase}
  
\fi
\IfFileExists{microtype.sty}{\usepackage{microtype}}{}
\ifxetex
  \usepackage[setpagesize=false, 
              unicode=false, 
              xetex]{hyperref}
\else
  \usepackage[unicode=true]{hyperref}
\fi
\hypersetup{breaklinks=true,
            bookmarks=true,
            pdfauthor={},
            pdftitle={},
            colorlinks=true,
            citecolor=blue,
            urlcolor=blue,
            linkcolor=magenta,
            pdfborder={0 0 0}}
\usepackage{amsfonts}
\usepackage{amsmath}
\usepackage{amssymb}
\usepackage{graphicx}

\usepackage{version}%
\usepackage{color}
\usepackage{amscd}     
\usepackage{mathrsfs}  

\newtheorem{theorem}{Theorem}[section]
\theoremstyle{plain}

\newtheorem{corollary}{Corollary}[section]

\newtheorem{definition}{Definition}[section]
\newtheorem{example}{Example}[section]

\newtheorem{proposition}{Proposition}[section]
\newtheorem{remark}{Remark}[section]

\numberwithin{equation}{section}
\excludeversion{comment1}

\newcommand{\N}{\mathbb{N}}
\newcommand{\Q}{\mathbb{Q}}
\newcommand{\R}{\mathbb{R}}
\newcommand{\X}{\mathbb{X}}
\newcommand{\Y}{\mathbb{Y}}
\newcommand{\Z}{\mathbb{Z}}
\newcommand{\cB}{\mathcal{B}}
\newcommand{\cC}{\mathcal{C}}

\newcommand{\cF}{\mathcal{F}}

\newcommand{\cH}{\mathcal{H}}

\newcommand{\sfX}{\mathsf{X}}
\newcommand{\fL}{\mathfrak{L}}
\newcommand{\fC}{\mathfrak{C}}
\newcommand{\fX}{\mathfrak{X}}

\newcommand{\be}{\begin{equation}}
\newcommand{\ee}{\end{equation}}
\newcommand{\inn}[2]{{\langle #1,#2 \rangle}}

\newcommand{\n}[1]{{\lVert #1 \rVert}}

\newcommand{\tn}[1]{{\left\vert\kern-0.25ex\left\vert\kern-0.25ex\left\vert #1 
    \right\vert\kern-0.25ex\right\vert\kern-0.25ex\right\vert}}

\newcommand{\abs}[1]{{\left| #1 \right|}}

\newcommand{\bl}{{\boldsymbol{l}}}

\newcommand{\bS}{{\boldsymbol{S}}}
\newcommand{\bq}{{\boldsymbol{q}}}

\newcommand{\nl}{\vskip 5pt\noindent}

\DeclareMathOperator{\id}{id}

\DeclareMathOperator{\Lip}{Lip}

\DeclareMathOperator{\Hom}{Hom}
\DeclareMathOperator*{\esup}{{ess\,sup}}
\DeclareMathOperator{\graph}{graph}

\title{Semi-Rings, Semi-Vector Spaces, and Fractal Interpolation}
\author{Peter R. Massopust}
\address{Department of Mathematics, Technical University of Munich, Garching by Munich, Germany}
\email{massopust@ma.tum.de}

\begin{document}

\maketitle

\begin{abstract}
In this paper, we introduce fractal interpolation on complete semi-vector spaces. This approach is motivated by the requirements of preservation of positivity or monotonicity of functions for some models in approximation and interpolation theory. The setting in complete semi-vector spaces does not requite additional assumptions but is intrinsically built into the frame work. For the purposes of this paper, fractal interpolation in the complete semi-vector spaces $C^+$ and $L_p^+$ is considered.
\end{abstract}
\section{Introduction}
This paper introduces fractal interpolation and solutions of the fractal interpolation problem in complete semi-vector spaces. This approach is motivated by the necessity in approximation theory to preserve positivity, monotonicity, or convexity of functions $f$. In each case, $\pm f^{k} \geq 0$, for $k=0,1,2$, respectively. Fractal interpolation in these situations have been carried out but needed additional requirements. The interested reader may find numerous articles to this behalf in the fractal literature.

The convenient setting of semi-vector spaces provides a means of only considering positive functions from the onset. Once it has been established that certain such semi-vector spaces are complete, they lend themselves to fractal interpolation and the application of the Banach fixed point theorem.

The structure of this paper is as follows. For the reader not familiar with the setting of semi-rings, semi-vector spaces, and related topics, we review and introduce these concepts in Section 2. This section forms the foundation for the remainder of this article. The two types of complete semi-vector spaces, namely $C^+I$ (continuous positive functions on a compact interval $I\subset\R$) and $L_p^+I$ (classes of positive functions in the Lebesgue spaces $L_p^+I$, $p\in [1,\infty]$) are defined here. For the sake of completeness, self-containment of the presentation, and the underlying foundational scope of fractal interpolation, we also provide the reader with the concepts of iterated function systems and fractal interpolation in Section 3. There, the fundamental object of a Read-Bajaktarevi\`c operator is introduced. In Section 5, we describe fractal interpolation in the complete semi-vector spaces $C^+I$ and $L_p^+I$ and derive conditions that guarantee the existence of fractal functions in each of these two spaces. We briefly dwell on fractal sequences and series in $C^+I$ and $L_p^+I$. The section closes by describing a semi-linear fractal operator on each of these two spaces and presenting some of its properties. A conclusion and future-research-directions complete the paper.
\section{Semi-Rings, Semi-Vector Spaces and Related Results}
In this section, we introduce those concepts from the areas of semi-rings and semi-vector spaces that are relevant for this paper. The interested reader is referred to \cite{LCLPZB,PS} and the references given therein for more details. The definitions and examples presented in this section can be found in \cite{GG,JMV,LCLPZB}.
\subsection{Definitions and Examples}
\begin{definition}
A set nonempty set $S$ endowed with two binary operations $+ :S\times S\to S$ (addition) and $\ast : S\times S\to S$ (multiplication) is called a semi-ring if
\begin{enumerate}
\item[(i)] $(S,+)$ is a commutative monoid;
\item[(ii)] $(S, \ast)$ is a semigroup;
\item[(iii)] multiplication $\ast$ is distributive with respect to $+$, i.e., $\forall\, x,y,z\in S$, $(x+y)\ast z = x\ast y + x\ast z$ and $x\ast (y+z) = x\ast y + x\ast z$.
\end{enumerate}
\end{definition}
To simplify notation, we write $S$ for $(S, +,\ast)$ unless it is necessary to avoid ambiguities. If multiplication $\ast$ is commutative then $S$ is called a commutative semi-ring, and if there exists a $1\in S$ such that $\forall\, x\in S$, $x\ast 1 = 1 \ast x = x$, then $S$ is called a semi-ring with (multiplicative) identity $1$. 
\begin{definition}
A semi-field is an ordered triple $K:=(K,+,\ast)$ which is a commutative semi-ring with $1$ satisfying the conditions
\begin{enumerate}
\item[(i)] $\forall\,x,y\in K$, $x+y=0$ implies $x=y=0$;
\item[(ii)] there are no multiplicative zero divisors in $K$.
\end{enumerate}
\end{definition}
\begin{example}
An example of a semi-field is the set $\R_0^+ :=\{x\in \R: x\geq 0\}$ endowed with the usual addition and multiplication of real numbers.
\end{example}
Now, we are ready to introduce the concept of a semi-vector space over a semi-field $K$.
\begin{definition}\label{def1}
Let $V$ be a nonempty set endowed with the binary operations $+:V\times V\to V$ (vector addition) and $\cdot: K\times V\to V$ (scalar multiplication). Then the ordered triple $(V,+,\cdot)$ is called a semi-vector space over the semi-field $K$ provided that
\begin{enumerate}
\item[(i)] $(V,+)$ is a commutative monoid satisfying the additive cancellation law $\forall\, u,v,w\in V$, $u+v = u+w$ implies $v = w$;
\item[(ii)] $\forall\,\alpha\in K\,\forall\,u,v\in K$: $\alpha(u+v) = \alpha u +\alpha v$;
\item[(iii)] $\forall\,\alpha,\beta\in K\,\forall\,v\in V$: $(\alpha + \beta)v = \alpha v + \beta v$;
\item[(iv)] $\forall\,\alpha,\beta\in K\,\forall\,v\in V$: $(\alpha\beta)v = \alpha(\beta v)$;
\item[(v)] $\forall\,v\in V$ and $1\in K$: $1\cdot v = v$.
\end{enumerate}
\end{definition}
To ease notation, we write $V$ for $(V,+,\cdot)$ unless this gives rise to confusion.
We remark that the semi-vector spaces introduced here are called \textit{regular} as they are required to satisfy Definition \ref{def1}(i). The (unique) element $0_V\in V$ with $v+0_V = v+0_V = v$, $\forall\,v\in V$ will be referred to as the zero or null vector of $V$.

Let $0_V\neq v\in V$. If $\exists u\in V$ such that $v+u = 0_V$ then $v$ is called \emph{symmetrizable}. A semi-vector space $V$ is called \emph{simple} if the only symmetrizable element is the zero vector $0_V$.

\begin{definition}\label{def2}
Suppose $V$ is a simple semi-vector space over the semi-field $\R_0^+$. A nonempty subset $B\subset V$ is called a semi-basis of $V$ if every $0_V\neq v\in V$ has a unique representation in the form $v = \sum\limits_{i\in I(v)} v\^{i} b_i$ where $v\^{i}\in \R^+ := \{x\in \R: x > 0\}$, $b_i\in B$ and $I(v)$ is a finite index set uniquely determined by $v$. The finite subset $B_v\subset B$ given by $B_v := \{b_i : i\in I(v)\}$ is uniquely determined by $v$. A semi-vector space is called semi-free if it admits a semi-basis. 
\end{definition}
A direct consequence of Definition \ref{def2} is the following corollary whose proof can be found in \cite[Corollary 1.7]{JMV}.
\begin{corollary}
All semi-bases in a semi-free semi-vector space $V$ have the same cardinality. This cardinality is called the semi-dimension of $V$.
\end{corollary} 
The example below list some semi-vector spaces. (See, for instance, also: \cite{LCLPZB}
\begin{example}\hfill
\begin{enumerate}
\item All vector spaces are semi-vector spaces. They are, however, not simple.
\item Let $n\in \N:= \{n\in \Z: n\geq 1\}$. The set $(\R_0^+)^n := \underset{i = 1}{\overset{n}{\sfX}}\R_0^+$ together with the usual component-wise addition and scalar multiplication is a semi-vector space over the semi-field $\R_0^+$, called the Euclidean semi-vector space.
\item For $m,n\in\N$, denote by $M_{n\times m}(\R_0^+)$ the set of all $n\times m$ matrices with entries from $\R_0^+$. Equipped with the usual matrix addition and matrix scalar multiplication, $M_{n\times m}(\R_0^+)$ forms a semi-vector space over $\R_0^+$.
\item Let $n\in \N$ and let $P_n [x]$ denote the set of all polynomials with coefficients from $\R_0^+$ together with the usual addition and scalar multiplication   is a semi-vector space over $\R_0^+$.
\end{enumerate}
\end{example}
\begin{definition}
A nonempty subset $U$ of a semi-vector space $(V,+,\cdot)$ over $\R_0^+$ is called a semi-subspace of $V$ if $U$ satisfies
\begin{enumerate}
\item[(i)]	$\forall\, u_1, u_2\in U$, $u_1+u_2\in U$;
\item[(ii)]	$\forall\,\alpha\in \R_0^+\,\forall\,u\in U$, $\alpha u\in U$.
\end{enumerate} 
\end{definition}
We note that the uniqueness of $0_V$ and the regularity of the semi-vector space $V$ imply that every semi-subspace $U$ of $V$ contains the zero vector.
The following are some examples of semi-subspaces.
\begin{example}\hfill
\begin{enumerate}
\item Let $\Q_0^+ := \{q\in \Q: q\geq 0\}$. The semi-vector space $\Q_0^+$ over $\Q_0^+$ is a semi-subspace of $\R_0^+$ over $\Q_0^+$.
\item The set of diagonal matrices in $M_{n\times n}(\R_0^+)$ is a semi-subspace of the latter.
\end{enumerate}
\end{example}
\begin{definition}\label{def3}
Let $V$ and $W$ be semi-vector spaces over $\R_0^+$ and let $T:V\to W$ be a map. $T$ is called a semi-linear transformation between $V$ and $W$ if
\begin{enumerate}
\item[(i)]	$\forall\,v_1, v_2\in V$: $T(v_1+v_2) = Tv_1 + Tv_2$;
\item[(ii)]	$\forall\,\lambda\in\R_0^+\,\forall\,v\in V$: $T(\lambda v) = \lambda Tv$.
\end{enumerate}
In case $W = V$, we call $T$ a semi-linear operator on $V$.
\end{definition}
\begin{remark}
The expression \emph{semi-linear} defined in Definition \ref{def3} is not to be confused with the one used in linear algebra and projective geometry where condition (ii) has the form: $\forall\,\lambda\in\R_0^+\,\forall\,v\in V$ it holds that $T(\lambda v) = \sigma(\lambda) Tv$, for some automorphism $\sigma \neq \id$ of the underlying field.
\end{remark}
\begin{definition}
Let $V$ and $W$ be semi-vector spaces over $\R_0^+$. A mapping $T:V\to W$ is called an affine semi-linear transformation if $T - T(0)$ is a semi-linear transformation $V\to W$.
\end{definition}
We note that the set $\Hom^+(V,W) := \{T:V\to W: T \textrm{ is semi-linear}\}$ is a semi-vector space over $\R_0^+$.

More results regarding semi-linear vector spaces and semi-linear transformation can be found in \cite[Section 3]{LCLPZB}.
\subsection{Complete Semi-Vector Spaces}
For the purposes of this paper, we need to introduce complete semi-vector spaces, that is, semi-vector spaces endowed with a norm, an inner product or a metric under which the space can be completed. The exposure here follows \cite[Section 3.1]{LCLPZB}.

In the remainder of this paper, all semi-vector spaces are taken over the semi-field $\R_0^+$ unless otherwise stated. As usual, we denote the multiplication in $\R_0^+$ and the scalar multiplication of the associated semi-vector space by $\cdot$ or simply use juxtaposition. Moreover, we denote by $\N_n$ the initial segment of $\N$ of length $n$, i.e., $\N_n := \{1, \ldots, n\}$, and set $\N_{0,n} := \N_n\cup\{0\}$.
\begin{definition}
Let $V$ be a semi-vector space. $V$ is called a normed semi-vector space if there exists a function $\n{\cdot}:V\to\R_0^+$ satisfying the conditions for a norm. $V$ is called a Banach semi-vector space if $V$ is complete in this norm.
\end{definition}
\begin{definition}
Let $V$ be a semi-vector space. $V$ is called a pre-Hilbert semi-vector space if there exists a bilinear function $\inn{\cdot}{\cdot}: V\times V\to \R_0^+$ on $V$ satisfying the conditions for an inner product. $V$ is called a Hilbert semi-vector space if $V$ is complete with respect to this inner product. 
\end{definition}
The next results shows that well-known norms on $\R^n$ are also norms on $(\R_0^+)^n$.
\begin{proposition}\label{prop1}
The Euclidean semi-vector space $(\R_0^+)^n$ becomes a normed semi-vector space when endowed with the function $\n{\cdot}: (\R_0^+)^n\to \R_0^+$,
\[
x := (x_1, \ldots, x_n)\mapsto \sqrt{\sum\limits_{i\in\N_n} x_i^2}.
\]
\end{proposition}
\begin{proof}
See \cite[Proposition 5]{LCLPZB}.
\end{proof}
The norm defined in Proposition \ref{prop1} will be referred to as the Euclidean or $\ell_2^+(\N_n)$ norm on $(\R_0^+)^n$.
\begin{proposition}\label{prop2}
Let $x := (x_1, \ldots, x_n), y := (y_1, \ldots, y_n)\in (\R_0^+)^n$. Set
\[
m_i:= m_i(x,y) := \min\{x_i,y_i\}, \quad M_i:= M_i(x,y):=\max\{x_i,y_i\}, \quad i\in \N_n.
\]
Define a function $d: (\R_0^+)^n\times (\R_0^+)^n\to \R_0^+$ by
\[
d(x,y) := \sqrt{\sum\limits_{i\in\N_n} (M_i - m_i)^2}.
\]
Then, $d$ is a metric on $(\R_0^+)^n$, called the Euclidean metric on $(\R_0^+)^n$.
\end{proposition}
\begin{proof}
See \cite[Proposition 6]{LCLPZB}.
\end{proof}
\begin{remark}
Although $\max\{x_i,y_i\} - \min\{x_i,y_i\} = \abs{x_i-y_i}$, we will continue using the $\max - \min$ difference instead of the absolute value. The reason will become clear in Section \ref{sec2.3}. 
\end{remark}
\begin{remark}\label{rem1}
It follows from Propositions \ref{prop1} and \ref{prop2}, that a semi-vector space $V$ is Banach if the norm generates a metric under which every Cauchy sequence in $V$ converges to an element of $V$.
\end{remark}
\begin{proposition}
On $(\R_0^+)^n$ define the bilinear mapping $\inn{\cdot}{\cdot}: (\R_0^+)^n\times (\R_0^+)^n\to \R_0^+$ by
\[
\inn{x}{y} := \sum_{i\in\N_n} x_i y_i,
\]
where $x := (x_1, \ldots, x_n), y := (y_1, \ldots, y_n)$. Then, $\inn{\cdot}{\cdot}$ defines an inner product on $(\R_0^+)^n$. Moreover, $\inn{\cdot}{\cdot}$ induces the Euclidean norm on $(\R_0^+)^n$.
\end{proposition}
\begin{proof}
See \cite[Propositions 7 and 8]{LCLPZB}.
\end{proof}
Similar to Remark \ref{rem1}, we note that an inner product on a semi-vector space $V$ induces and norm and thus a metric on $V$. Hence, $V$ becomes Hilbert if every Cauchy sequence in $V$ converges in the metric to an element of $V$.

The next results relate norms and semi-linear transformations.
\begin{definition}
Let $V$ and $W$ be semi-vector spaces and $T:V\to W$ a semi-linear transformation. The mapping $T$ is called bounded if $\exists\,M > 0$ such that $\forall\,v\in V$, $\n{Tv} \leq M \n{v}$.
\end{definition}
The above definition allows to define a norm on bounded elements of $\Hom^+(V,W)$. To this end, set
\be\label{eq1}
\n{T} := \sup_{0_V\neq v\in V} \frac{\n{Tv}}{\n{v}}.
\ee
The fact that Eqn. \eqref{eq1} defines indeed a norm is proven in \cite[Proposition 14]{LCLPZB}.
\subsection{Some Complete Semi-Vector Spaces}\label{sec2.3}
Here, we introduce some complete (in the underlying metric) semi-vector spaces that are relevant to this paper. More details can be found in \cite{LCLPZB} and the references given therein.
\begin{theorem}
Let $p\in [1, \infty]$. Define
\be
\ell_p^+ := \ell_p^+ (\N) := \begin{cases}
\left\{a:\N\to\R_0^+ : \sum\limits_{n\in\N} \abs{a(n)}^p < \infty\right\}, & p\in [1, \infty);\\
\left\{a:\N\to\R_0^+ : \sup\limits_{n\in\N} \abs{a(n)} < \infty\right\}, & p = \infty.
\end{cases}
\ee
For $a,b\in \ell_p^+$, set
\be
d_p(a,b) := \begin{cases}
\left(\sum\limits_{n\in\N} (M_n - m_n)^p\right)^{1/p}, & p\in [1, \infty);\\
\sup\limits_{n\in\N}\, (M_n - m_n), & p = \infty,
\end{cases} 
\ee
where
\[
m_n= \min\{a_n,b_n\}, \quad M_n:= \max\{x_n,y_n\}, \quad n\in \N.
\]
Then, the semi-spaces $\ell_p^+$ endowed with the metric $d_p$ are complete.
\end{theorem}
\begin{proof}
See \cite[pp. 10--13]{LCLPZB}.
\end{proof}
\begin{definition}
Suppose $[a,b]$ is a nonempty interval in $\R_0^+$. Define
\be
C^+[a,b] := \left\{f\in C[a,b] : f: [a,b]\to\R_0^+ \right\}
\ee 
and let $d: C^+[a,b]\times C^+[a,b]\to \R_0^+$ be given by
\be\label{dfg}
d(f,g) := \max_{x\in [a,b]} \left\{\max\{f(x), g(x)\} - \min\{f(x), g(x)\}\right\}.
\ee
\end{definition}
Note that for $f\in C^+[a,b]$, 
\[
\n{f}_\infty := \sqrt{d(f,f)} = \sqrt{\max_{x\in [a,b]} \left\{\max\{f(x)\} - \min\{f(x)\}\right\}}
\]
defines a norm on $C^+[a,b]$.
\begin{theorem}
The mapping $d$ defined by Eqn. ~\eqref{dfg} is a metric on the semi-space $C^+[a,b]$. When endowed with this metric, $C^+[a,b]$ becomes a complete metrizable space.
\end{theorem}
\begin{proof}
The proof can be found in \cite[Theorem 7]{LCLPZB}.
\end{proof}
\begin{remark}\label{rem2.4}
Note that despite $\max\{f(x), g(x)\} - \min\{f(x), g(x) = \abs{f(x) - g(x)}$, $C^+[a,b]$ is \emph{not} a closed subspace of $C[a,b]$ as the scalars are only taken from $\R_0^+$.
\end{remark}
For $p\in [0,\infty]$ and a nonempty interval $[a,b]\subset\R_0^+$, we now introduce the semi-space 
\be
L_p^+[a,b] := \{f\in L_p[a,b] : f:[a,b]\to\R_0^+\},
\ee
where $L_p[a,b]$ denotes the Lebesgue space of functions  with exponent $p$ on $[a,b]$, and endow it with the mapping $d_p:L_p^+[a,b]\times L_p^+[a,b]\to\R_0^+$,
\be\label{Lpmetric}
d_p (f,g) := \begin{cases}
\displaystyle{\int\limits_{[a,b]}} \left(\max\{f(t),g(t)\} - \min\{f(t),g(t)\}\right)^p dt, & p\in [1,\infty);\\
\esup\limits_{t\in [a,b]} \left(\max\{f(t),g(t)\} - \min\{f(t),g(t)\}\right), & p =\infty,
\end{cases}
\ee
with $dt$ denoting Lebesgue measure.

Note that the mappings $d_p$ are well-defined as the $\max$ and $\min$ of measurable functions is again a measurable function and the values of $f$ and $g$ being $L_p$-functions are finite almost everywhere.
\begin{theorem}
The mapping $d_p$ defined in \eqref{Lpmetric} is a metric on $L_p^+[a,b]$ and with it $L_p^+[a,b]$ becomes a complete metrizable space.
\end{theorem}
\begin{proof}
The verification that $d_p$ is a metric is straight-forward and the proof that $L_p^+[a,b]$ is complete with respect to $d_p$ is similar to the one given in \cite[Theorem 7]{LCLPZB}. We leave the details to the reader.
\end{proof}

Replacing Lebesgue measure in the definition of \eqref{Lpmetric} by counting measure, we obtain the semi-spaces $\ell_p^+$.
\begin{remark}
Remark \ref{rem2.4} also applies to the spaces $L_p^+[a,b]$: The latter is \emph{not} a closed subspace of $L_p[a,b]$.
\end{remark}
Note that for $f\in L_p^+[a,b]$, 
\[
\n{f}_p := \sqrt{d_p(f,f)} = \sqrt{\int_{[a,b]} \left(\max_{x\in [a,b]} \left\{\max\{f(x)\} - \min\{f(x)\}\right\}\right)^p}
\]
defines a norm on $L_p^+[a,b]$ (with the usual modification for $p=\infty$).
\section{Preliminaries on Iterated Function Systems}
Here, we state the definition of iterated function system and its attractor.

Let $(\X,d)$ be a complete metric space. For a map $f: \X \to \X$, define the Lipschitz constant associated with $f$ by
\[
\Lip (f) = \sup_{x,y \in \X, x \neq y} \frac{d_\X\big(f(x),f(y)\big)}{d_\X(x,y)}.
\]
The map $f$ is said to be Lipschitz if $\Lip (f) < + \infty$ and a contraction on $\X$ if $\Lip (f) < 1$.

\begin{definition}
Let $1<n\in\N$ and let $\cF := \{f_1, \ldots, f_n\}$ be a finite family of contractions on the complete metric space $(\X,d_\X)$. Then the pair $(\X, \cF)$ is called a contractive iterated function system (IFS) on $\X$.
\end{definition}
As this article deals exclusively with contractive IFSs, we drop the adjective ``contractive'' in the following.

With an IFS $(\X,\cF)$ and its point maps $f\in \cF$, we associate a set-valued map, also denoted by $\cF$, as follows. Let $(\cH(\X), h)$ be the hyperspace of all nonempty compact subsets of $\X$ endowed with the Hausdorff-Pompeiu metric 
\[
h (S_1, S_2) := \max\{d(S_1, S_2), d(S_2, S_1)\},
\]
where $d(S_1,S_2) := \sup\limits_{x\in S_1} d(x, S_2) := \sup\limits_{x\in S_1}\inf\limits_{y\in S_2} d(x,y)$. 

Define $\cF: \cH(\X)\to \cH(\X)$ by (cf. \cite{B1,H})
\be\label{1.1}
\cF (S) := \bigcup_{i=1}^n f_i (S).
\ee
For contractive mappings $f\in \cF$, the set-valued map $\cF$ defined by \eqref{1.1} is also a Lipschitz map on $\cH(\X)$ with Lipschitz constant $\Lip (\cF) = \max \{\Lip (f_i) : i\in\N_n\} < 1$. Also, the completeness of $(\X,d)$ implies the completeness of $(\cH(\X), h)$. 

%
\begin{definition}
The unique fixed point $A\in \cH(\X)$ of the contractive set-valued map $\cF$ is called the attractor of the IFS $(\X,\cF)$.
\end{definition}

Note that since $A$ satisfies the \emph{self-referential equation}
\be
A = \cF(A) = \bigcup_{i=1}^n f_i (A),
\ee
the attractor is in general a fractal set.

It follows directly from the proof of the Banach Fixed Point Theorem that the attractor $A$ is obtained as the limit (in the Hausdorff-Pompeiu metric) of the iterative process $A_k := \cF(A_{k-1})$, $k\in \N$:
\be
A = \lim_{k\to\infty} A_k = \lim_{k\to\infty} \cF^k (A_0),
\ee
for an arbitrary $A_0\in \cH(\X)$. Here, $\cF^k$ denotes the $k$-fold composition of $\cF$ with itself.

%
\section{Fractal Interpolation and Fractal Functions}\label{sect3}
We briefly recall some definitions and properties from fractal interpolation and fractal function theory. In the following, $(\X, d_\X)$ denotes a complete metrizable space.
\subsection{Fractal Interpolation on $\cB(\X, \Y)$}
Suppose a finite family $\{l_i\}_{i = 1}^{n}$ of injective contractions $\X\to \X$ generating a partition of $\X$ in the sense that
\begin{align}
&\X = \bigcup_{i=1}^n l_i(\X);\label{part1}\\
&l_i(\X)\cap l_j(\X) = \emptyset, \quad\forall\;i, j\in \N_n, i\neq j.\label{part2} 
\end{align}

Given another complete metrizable space $(\Y,d_\Y)$ with metric $d_\Y$, a mapping $g:\X\to \Y$ is called \emph{bounded} (with respect to the metric $d_\Y$) if there exists an $M> 0$ so that for all $x_1, x_2\in \X$, $d_\Y(g(x_1),g(x_2)) < M$.

Recall that the set $\cB(\X, \Y) := \{g : \X\to \Y : \text{$g$ is bounded}\}$ when endowed with the metric 
\be\label{d}
d(g,h): = \displaystyle{\sup_{x\in \X}} \,d_\Y(g(x), h(x))
\ee
becomes a complete metrizable space.

\begin{remark}
Under the usual addition and scalar multiplication of functions, the space $\cB(\X,\Y)$ becomes actually a metric linear space, i.e., a vector space under which the operations of vector addition and scalar multiplication are continuous. (See, for instance, \cite{R}.)
\end{remark}

For $i\in\N_n$, let $F_i: \X\times \Y \to \Y$ be a mapping which is uniformly contractive in the second variable, i.e., there exists a $c\in [0,1)$ so that for all $y_1, y_2\in \Y$
\be\label{scon}
d_\Y (F_i(x, y_1), F_i(x, y_2)) \leq c\, d_\Y (y_1, y_2), \quad\forall x\in \X,\,\forall i \in\N_n.
\ee
Define an operator $T: \cB(\X,\Y)\to \Y^{\X}$, by
\be\label{RB}
T g (x) := \sum\limits_{i=1}^n F_i (l_i^{-1} (x), g\circ l_i^{-1} (x))\,\chi_{l_i(\X)}(x), 
\ee
where $\chi_M$ denotes the characteristic function of a set $M$. Such operators are referred to as \emph{Read-Bajractarevi\'c (RB)} operators. The operator $T$ is well-defined and since $g$ is bounded and each $F_i$ contractive in the second variable, $T g\in \cB(\X,\Y)$.

Equivalently, \eqref{RB} can also be written in the form
\be\label{3.3}
(T g \circ l_i) (x) := F_i (x, g(x)),\quad x\in \X, \;i\in\N_n. 
\ee

Moreover, \eqref{scon} implies that $T$ is contractive on $\cB(\X, \Y)$:
\begin{align}\label{estim}
d(T g, T h) & = \sup_{x\in \X} d_\Y (T g (x), T h (x))\nonumber\\
& = \sup_{x\in \X} d_\Y (F(l_i^{-1} (x), g(l_i^{-1} (x))), F(l_i^{-1} (x), h(l_i^{-1} (x))))\nonumber\\
& \leq c\sup_{x\in \X} d_\Y (g\circ l_i^{-1} (x), h \circ l_i^{-1} (x)) \leq c\, d_\Y(g,h).
\end{align}
To achieve notational simplicity, we have set $F(x,y):= \sum\limits_{i=1}^n F_i (x, y)\,\chi_{\X}(x)$ in the above equation. 

Therefore, by the Banach Fixed Point Theorem, $T$ has a unique fixed point $f^*$ in $\cB(\X,\Y)$. This unique fixed point is called the \emph{bounded fractal function} generated by $T$ and it satisfies the \emph{self-referential equation}
\be\label{eq2.8}
f^*(x) = \sum\limits_{i=1}^n F_i (l_i^{-1} (x), f^*\circ l_i^{-1} (x))\,\chi_{l_i(\X)}(x),
\ee
or, equivalently,
\be
f^*\circ l_i (x) = F_i (x, f^*(x)),\quad x\in \X, \;i\in\N_n.
\ee
Per proof of the Banach Fixed Point Theorem, the fixed point $f^*\in \cB(\X,\Y)$ is obtained as the limit of the sequence of mappings
\be\label{3.9}
T^k (f_0) \xrightarrow{\,\,d\,\,} f^*, \quad\text{as $k\to\infty$},
\ee
where $f_0\in \cB(\X,\Y)$ is arbitrary.

Next, we would like to consider a special choice for the mappings $F_i$. To this end, we require the concept of an $F$-space. We recall that a metric $d:\Y\times \Y\to \R$ is called \textit{complete} if every Cauchy sequence in $\Y$ converges with respect to $d$ to a point of $\Y$, and \textit{translation-invariant} if 
\[
d(x+a,y+a) = d(x,y), \quad\text{for all $x,y,a\in \Y$}.
\]

Now assume that $\Y$ is an \emph{${F}$-space}, i.e., a topological vector space whose topology is induced by a complete translation-invariant metric $d$, and in addition that this metric is homogeneous. This setting allows us to consider mappings $F_i$ of the form
\be\label{specialv}
F_i (x,y) :=q_i (x) + S_i (x) \,y,\quad i\in\N_n,
\ee
where $q_i \in \cB(\X,\Y)$ and $S_i : \X\to \R$ is a function.

As the metric $d_\Y$ is homogeneous, the mappings \eqref{specialv} satisfy condition \eqref{scon} provided that the functions $S_i$ are bounded on $\X$ with bounds in $[0,1)$. For then
\begin{align*}
d_\Y (q_i (x) + S_i (x) \,y_1,&q_i (x) + S_i (x) \,y_2) = d_\Y(S_i (x) \,y_1,S_i (x) \,y_2) \\
& = |S_i(x)| d_\Y (y_1, y_2) \leq \|S_i\|_{\infty}\, d_\Y (y_1, y_2) \leq s\,d_\Y (y_1, y_2).
\end{align*}
Here, $\|\cdot\|_{\infty}$ denotes the supremum norm and $s := \max\{\|S_i\|_{\infty}:$ $i\in\N_n\}$. Henceforth, we will assume that all functions $S_i$ are bounded above by $s\in [0,1)$.

With the choice \eqref{specialv}, the RB operator $T$ becomes an affine operator on $\cB(\X,\Y)$ of the form
\begin{align}\label{T}
T g & = \sum_{i=1}^n (q_i\circ l_i^{-1}) \chi_{l_i(\X)} + \sum_{i=1}^n (S_i\circ l_i^{-1}) \cdot (g\circ l_i^{-1}) \chi_{l_i(\X)}\\
& = T(0) + \sum_{i=1}^n (S_i\circ l_i^{-1}) \cdot (g\circ l_i^{-1}) \chi_{l_i(\X)}.
\end{align}
Next, we exhibit the relation between the graph $G(f^*)$ of the fixed point $f^*$ of the operator $T$ given by \eqref{RB} and the attractor of an associated contractive IFS. 

To this end, consider the complete metric space $\X\times \Y$ and define mappings $w_i:\X\times \Y\to \X\times \Y$ by
\be\label{wn}
w_i (x, y) := (l_i (x), F_i (x,y)), \quad i\in\N_n.
\ee
Assume that the mappings $F_i$ in addition to being uniformly contractive in the second variable are also uniformly Lipschitz continuous in the first variable, i.e., that there exists a constant $L > 0$ so that for all $y\in \Y$,
\[
d_\Y(F_i(x_1, y),F_i(x_2, y)) \leq L \, d_\X (x_1,x_2), \quad\forall x_1, x_2\in \X,\quad\forall i = 1, \ldots, n.
\]
Denote by $a:= \max\{a_i : i\in\N_n\}$ the largest of the contractivity constants of the $l_i$ and let $\theta := \frac{1-a}{2L}$. Then the mapping $d_\theta : (\X\times \Y)\times (\X\times \Y) \to \R$ given by
\[
d_\theta := d_\X + \theta\,d_\Y
\]
is a metric on $\X\times \Y$ compatible with the product topology on $\X\times \Y$.

The next theorem is a special case of a result presented in \cite{bhm}.

\begin{theorem}\label{thm3.1}
The family $\cF_T := (\X\times \Y, w_1, w_2, \ldots, w_n)$ is a contractive IFS in the metric $d_\theta$ and the graph $G(f^*)$ of the fractal function $f^*$ generated by the RB operator $T$ given by \eqref{RB} is the unique attractor of $\cF_T$. Moreover, 
\be\label{GW}
G(T g) = \cF_T (G(g)),\quad\forall\,g\in B(\X,\Y),
\ee
where $\cF_T$ denotes the set-valued operator \eqref{1.1}.
\end{theorem}

Equation \eqref{GW} can be represented by the following commutative diagram
\be\label{diagram}
\begin{CD}
\X\times \Y @>\cF_T>> \X\times \X\\
@AAGA                  @AAGA\\
\cB(\X,\Y) @>T>>  \cB(\X,\Y)
\end{CD}
\ee
\nl
where $G$ is the mapping $\cB(\X,\Y)\ni g\mapsto G(g) = \{(x, g(x)) : x\in \X\}\in \X\times \Y$.

On the other hand, suppose that $\cF = (\X\times \Y, w_1, w_2, \ldots, w_n)$ is an IFS whose mappings $w_i$ are of the form \eqref{wn} where the functions $l_i$ are contractive injections satisfying \eqref{part1} and \eqref{part2}, and the mappings $F_i$ are  uniformly Lipschitz continuous in the first variable and uniformly contractive in the second variable. Then we can associate with the IFS $\cF$ an RB operator $T_\cF$ of the form \eqref{RB}. The attractor $A_\cF$ of $\cF$ is then the graph $G(f)$ of the fixed point $f$ of $T_\cF$. (This was the original approach in \cite{B2} to define a fractal interpolation function on a compact interval in $\R$.) The commutativity of the diagram \eqref{diagram} then holds with $\cF_T$ replaced by $\cF$ and $T$ replaced by $T_\cF$.

We now specialize even further and choose an arbitrary $f\in\cB(\X,\Y)$ and a bounded linear operator $L: \cB(\X,\Y)\to\cB(\X,\Y)$. Set (see also \cite{NC})
\be\label{q}
q_i := f\circ l_i - S_i\cdot Lf.
\ee
The RB operator $T$ then reads 
\be\label{eq3.17}
T g = f + (S_i\circ l_i^{-1})\cdot (g-L f)\circ l_i^{-1}, \quad \textrm{on} \quad l_i(\X),\, i\in \N_n. 
\ee 
and, under the assumption that $s < 1$ its unique fixed point $f^*\in \cB(\X,\Y)$ satisfies the self-referential equation
\be\label{3.15}
f^* = f + (S_i \circ l_i^{-1})\cdot ({f^*}-Lf)\circ l_i^{-1}, \quad \textrm{on} \quad l_i(\X),\, i\in \N_n.
\ee

In the case of univariate fractal interpolation on the real line with $\X:= [x_0, x_n]$, $-\infty < x_0 < x_n < +\infty$, $Lf$ can be chosen to be the affine function $\beta$ whose graph connects the points $(x_0, f(x_0))$ and $(x_n, f(x_n))$. 
\subsection{Fractal Interpolation on $\cC(\X, \Y)$}\label{sect4.2}
Here, we briefly consider fractal interpolation in the complete metric space of continuous functions $(\cC(\X, \R), d)$ instead of $(\cB(\X,\R),d)$. For this purpose, the following changes in the construction ensure that the RB operator \eqref{eq3.17} maps $(\cC(\X, \R)$ into itself. 

To this end, let $x_0 < x_1 < \cdots < x_{n-1} < x_n$ be an increasing sequence of points in $[x_0,x_n]$ and define contractive homeomorphisms $l_i :[x_0,x_n] \to [x_{i-1},x_i]$ such that $x_{i-1} = l_i(x_0)$ and $x_i = l_i (x_n)$. Further assume that $f$, $Lf$, and $S_i$ are continuous on $[x_0,x_n]$, that
\[
Lf (x_0) := f(x_0)\quad\textrm{and}\quad Lf (x_n) := f(x_n),
\]
and impose the (interior) join-up conditions
\be\label{joinup}
Tf (x_j-) = Tf(x_j+), \quad j \in \N_{n-1}.
\ee
Then, the fixed point $f^*$ of the RB operator defined in Eqn. ~\eqref{eq3.17} will be a continuous function whose graph interpolates the set $\{(x_j, f(x_j)) : j = 0, 1,\ldots, n\}$. Such functions are usually referred to as \emph{fractal interpolation functions} \cite{B2,H}. As the RB operator is the same at each level of recursion \eqref{3.9}, we refer to this as \emph{stationary fractal interpolation}. For the non-stationary setting, i.e., the case where the RB operator is level-dependent, we refer the interested reader to \cite{M}. 

For fractal interpolation in more general function spaces, we list \cite{M2} and for the particular choice of $q$ in \eqref{q} \cite{NC} as references.
\section{Fractal Interpolation in Complete Semi-Vector Spaces}
In this section, we consider fractal interpolation in the semi-spaces $C^+$ and $L_p^+$, $p\in [0,\infty]$. To this end, let $I:=[x_0,x_n]\subset\R_0^+$ be a given nonempty interval and let $L:C^+I\to C^+I$ a bounded semi-linear operator in the sense of Definition~\ref{def3}. 

Suppose that $(x_j)_{j\in\N_{0,n}}$ is a strictly increasing sequence of real numbers. Further suppose that for each $i\in\N_n$, $l_i$ is a contractive diffeomorphism $[x_0,x_n]\to [x_{i-1},x_i]$ with 
\be\label{l}
l_i(x_0) = x_{i-1}\quad \text{and}\quad l_i(x_n) = x_i. 
\ee
Setting
\be\label{part}
I_i:= \begin{cases}
[x_{i-1}, x_i), & i\in\N_{n-1}\\
[x_{n-1},x_n], & i = n
\end{cases},
\ee
we have that $I = \coprod\limits_{i=1}^n I_i$ with $\coprod$ denoting disjoint union. For later purposes, we also assume that 
\be\label{Dl}
D l_i\in L_\infty^+I,\qquad i\in \N_n,
\ee
where $D$ denotes the derivative with respect to the independent variable.
\subsection{Fractal Interpolation in $L_p^+I$}
Let $\{l_i : i\in\N_n\}$ be given as in \eqref{l}. Define a semi-affine RB operator $T: L_p^+I\to (\R_0^+)^I$ of the form \eqref{T}
\be\label{TL}
T g : = \sum_{i=1}^n (q_i\circ l_i^{-1}) \chi_{l_i} + \sum_{i=1}^n (S_i\circ l_i^{-1}) \cdot (g\circ l_i^{-1}) \chi_{l_i},
\ee
where $q_i\in L_p^+I$ and $S_i\in L_\infty^+I$. The requirement on $S_i$ is based on the fact that $L_\infty \cdot L_p \subset L_p$. Hence, $Tg\in L_p^+I$.

To derive conditions for $T$ to be contractive, consider $g,h\in L_p^+$. 
\begin{align*}
d_p(Tg, Th) &= \int_I \left(\max\{Tg(x), Th(x)\} - \min\{Tg(x), Th(x)\}\right)^p dx\\
& = \sum_{i=1}^n \int_{I_i}  \left(\max\{(q_i\circ l_i^{-1})(x) + (S_i\circ l_i^{-1})(x) \cdot (g\circ l_i^{-1})(x), \right.\\
& \qquad \left. (q_i\circ l_i^{-1})(x) + (S_i\circ l_i^{-1})(x) \cdot (h\circ l_i^{-1})(x)\} - \right. \\
& \qquad - \left.\min\{(q_i\circ l_i^{-1})(x) + (S_i\circ l_i^{-1})(x) \cdot (g\circ l_i^{-1})(x), \right.\\
& \qquad \left. (q_i\circ l_i^{-1})(x) + (S_i\circ l_i^{-1})(x) \cdot (h\circ l_i^{-1})(x)\}\}\right)^p dx\\
& = \sum_{i=1}^n \int_{I_i}  \left(\max\{(S_i\circ l_i^{-1})(x) \cdot (g\circ l_i^{-1})(x), \right.\\
& \qquad \left. (S_i\circ l_i^{-1})(x) \cdot (h\circ l_i^{-1})(x)\} - \right. \\
& \qquad - \left.\min\{(S_i\circ l_i^{-1})(x) \cdot (g\circ l_i^{-1})(x), \right.\\
& \qquad \left. (S_i\circ l_i^{-1})(x) \cdot (h\circ l_i^{-1})(x)\}\}\right)^p dx\\
& \leq \sum_{i=1}^n \n{D l_i\cdot S_i^p}_\infty d_p(g,h),
\end{align*}
where we used the substitution $x\mapsto l_i (x)$ to obtain the last inequality. Note that by the assumptions on $Dl_i$ and $S_i$ and the fact that $L_\infty$ is a Banach algebra under pointwise multiplication,  $\n{D l_i\cdot S_i^p}_\infty$ is finite.

Thus, $T$ is a contractive semi-affine operator on $L_p^+I$ if  $\sum\limits_{i=1}^n \n{D l_i \cdot S_i^p}_\infty < 1$. 
\begin{theorem}\label{thm5.1}
Let contractive homeomorphism $(l_i)_{i\in\N_n}$ be defined as in \eqref{l} and satisfying \eqref{Dl}. Suppose $(q_i)_{i\in\N_n}\subset L_p^+I$ and $(S_i)_{i\in\N_n}\subset L_\infty^+I$. The semi-affine RB operator given by \eqref{TL} is contractive on $L_p^+I$ provided
\be\label{Lcond}
\sum_{i=1}^n \n{D l_i\cdot S_i^p}_\infty < 1.
\ee
The unique fixed point $f^*$ of $T$ satisfies the self-referential equation
\[
f^* = \sum_{i=1}^n (q_i\circ l_i^{-1}) \chi_{l_i} + \sum_{i=1}^n (S_i\circ l_i^{-1}) \cdot (f^*\circ l_i^{-1}) \chi_{l_i}
\]
and $\graph f^*$ is the attractor of the IFS associated with $T$ and thus in general a fractal set.
\end{theorem}
\begin{proof}
The statements follow from the above arguments and Theorem \ref{thm3.1}.
\end{proof}
Note that the fixed point $f^*$ depends on the the partition $(I_i)_{i\in\N_n}$, i.e., the mappings $\bl:=(l_i)_{i\in\N_n}$, the scaling functions $\bS:=(S_i)_{i\in\N_n}$, and the functions $\bq:=(q_i)_{i\in\N_n}$. Therefore, a more accurate notation for $f^*$ should be $f^*(\bl,\bS,\bq)$. Unless necessary, we suppress the dependence of the fixed point on $\bl$, $\bS$, and $\bq$. Should $f^*$ only depend on $\bullet\in\{\bl,\bS,\bq\}$ then we express this dependence as $f^*(\bullet)$.
\begin{definition}
The unique fixed point $f^*(\bl,\bS,\bq)$ in Theorem \ref{thm5.1} is called a fractal function of class $L_p^+$. We denote the collection of all such fractal functions by $\fL_p^+ := \fL_p^+(\bl,\bS,\bq)$.
\end{definition}

\begin{remark}
If one considers a semi-affine RB operator of the form \eqref{eq3.17} with some bounded semi-linear operator $L$ then additional conditions need to be imposed to guarantee that such an RB operator maps $L_p^+I$ into itself:
\begin{enumerate}
\item For $f\in L_p^+I$, $L$ must be such that $Lf\in L_p^+I$. This may be achieved by choosing for instance $Lf := v \cdot f$, for some $v\in L_\infty^+I$.
\item The terms $q_i = f\circ l_i - S_i\cdot Lf$ must be nonnegative. A necessary condition could be $f \geq \n{S_i \cdot Lf}_\infty$ on $I$.
\end{enumerate}
Once these two conditions are satisfied, the statements in Theorem \ref{thm5.1} also hold for this setting.
\end{remark}

In the special case when $(l_i)_{i\in\N_n}$ are affine functions, i.e., have the form $l_i = a_i (\cdot) + b_i$, where the coefficients $a_i$ and the translates $b_i$ are determined by \eqref{l}, condition \eqref{Lcond} is replaced by
\be\label{affinecond}
\max_{i\in\N_n}\{\n{S_i^p}_\infty\} < 1.
\ee
This is seen as follows: The Lipschitz constant for $D l_i$ is given by $a_i = \frac{x_i - x_{i-1}}{x_n - x_0}$ and therefore $\sum\limits_{i=1}^n \n{D l_i\cdot S_i^p}_\infty \leq \max\limits_{i\in\N_n}\{\n{S_i^p}_\infty\} \sum\limits_{i=1}^n \frac{x_i - x_{i-1}}{x_n - x_0} = \max\limits_{i\in\N_n}\{\n{S_i^p}_\infty\}$.

If in addition $(S_i)_{i\in\N_n} \subset\R_0^+$, then \eqref{affinecond} simply reduces to
\[
\max_{i\in\N_n} S_i^p < 1.
\]

Suppose that $\bl$ and $\bS$ are fixed. Then, the following result holds.
\begin{proposition}\label{prop5.1}
The mapping $\vartheta: (L_p^+)^n \to \fL_p^+(\bq)$, $\bq\mapsto f^*(\bq)$ is a semi-linear isomorphism. 
\end{proposition}
\begin{proof}
The result follows from the fact that $L_p^+I$ is a semi-vector space and that the fixed point $f^*$ is unique.
\end{proof}
\subsubsection{Sequences and Series of Fractal Functions of Type $L_p^+$}
Let $(f_m)_{m\in\N}\subset \fL_p^+$ be an $L_\infty^+$--bounded sequence of fractal functions with the self-referential equations
\be\label{eqseq}
f_m = \sum_{i\in\N_n} q_{m,i}\circ l_i^{-1} \chi_{I_i} + \sum_{i\in\N_n} (S_{m,i} \circ l_i^{-1})\cdot (f_m\circ l_i^{-1}) \chi_{I_i},\quad m\in \N,
\ee
where $(S_{m,i})_{m\in\N}\subset L_\infty^+I$ and $(q_{m,i})_{m\in\N}\subset L_p^+I$ are Cauchy sequences.
\begin{theorem}\label{thm5.2}
The bounded sequence \eqref{eqseq} of fractal functions converges in $L_p^+$ to a fractal function $f\in\fL_p^+$ given by the self-referential equation
\be\label{limseq}
f = \sum_{i\in\N_n} q_{i}\circ l_i^{-1} \chi_{I_i} + \sum_{i\in\N_n} (S_{i} \circ l_i^{-1})\cdot (f\circ l_i^{-1}) \chi_{I_i},
\ee
where $q_i = \lim\limits_{m\to\infty} q_{m,i}$ and $S_i = \lim\limits_{m\to\infty} S_{m,i}$.
\end{theorem}
\begin{proof}
For $k,m\in\N$, we compute, using properties of $\max$ and $\min$,
\begin{align*}
d_p (f_m,f_k) & \leq d_p\left(\sum_{i\in\N_n} q_{m,i}\circ l_i^{-1} \chi_{I_i},\sum_{i\in\N_n} q_{k,i}\circ l_i^{-1} \chi_{I_i}\right)\\
& \quad + d_p\left(\sum_{i\in\N_n} (S_{m,i} \circ l_i^{-1})\cdot (f_m\circ l_i^{-1}) \chi_{I_i},\sum_{i\in\N_n} (S_{k,i} \circ l_i^{-1})\cdot (f_k\circ l_i^{-1}) \chi_{I_i}\right).
\end{align*}
The first term reduces to
\begin{align*}
& \sum_{i\in\N_n} \int_{I_i} \left(\max\left\{q_{m,i}\circ l_i^{-1} ,q_{k,i}\circ l_i^{-1} \right\} - \min\left\{q_{m,i}\circ l_i^{-1} ,q_{k,i}\circ l_i^{-1} \right\}\right)^p dx\\
& = \sum_{i\in\N_n} \n{Dl_i}_\infty\, d_p(q_{m,i}, q_{k,i}).
\end{align*}
The second term becomes
\begin{align*}
& \sum_{i\in\N_n} \int_{I_i} \left(\max\left\{(S_{m,i} \circ l_i^{-1})\cdot (f_m\circ l_i^{-1}), (S_{k,i} \circ l_i^{-1})\cdot (f_k\circ l_i^{-1}) \right\}\right.\\ 
& \quad\left. - \min\left\{(S_{m,i} \circ l_i^{-1})\cdot (f_m\circ l_i^{-1}),(S_{k,i} \circ l_i^{-1})\cdot (f_k\circ l_i^{-1}) \right\}\right)^p dx\\
& = \sum_{i\in\N_n}\int_{I}  \abs{Dl_i} \left(\max\left\{S_{m,i}\cdot f_m, S_{k,i}\cdot f_k \right\} - \min\left\{S_{m,i}\cdot f_m ,S_{k,i}\cdot f_k \right\}\right)^p dx.
\end{align*}
Adding and subtracting the term $S_{m,i}f_k$, we estimate
\begin{align*}
&\int_I \abs{Dl_i} \left(\max\left\{S_{m,i}\cdot f_m, S_{k,i}\cdot f_k \right\} - \min\left\{S_{m,i}\cdot f_m ,S_{k,i}\cdot f_k \right\}\right)^p dx\\
& \quad\leq \n{Dl_i\cdot S_{m,i}^p}_\infty d_p(f_m,f_k) + \n{Dl_i \cdot f_k^p}_\infty d_p(S_{m,i},S_{k,i}).
\end{align*}
Substitution and reordering terms gives
\[
d_p (f_m,f_k) = \frac{\sum\limits_{i\in\N_n}\left[ \n{Dl_i}_\infty\, d_p(q_{m,i}, q_{k,i} + \n{Dl_i \cdot f_k^p}_\infty d_p(S_{m,i},S_{k,i})\right]}{1-\sum\limits_{i\in\N_n} \n{D l_i\cdot S_i^p}_\infty}
\]
By the hypotheses stated in the theorem, it now follows that $(f_m)_{m\in\N}$ is a Cauchy sequence and thus possesses a limit in $L_p^+I$. Taking the limit as $m\to\infty$ in \eqref{eqseq} proves \eqref{limseq}.
\end{proof}

Now suppose that for each $i\in\N_n$,
\[
q_{i} := \sum_{k\in\N} q_{k,i}\quad\text{and}\quad S_i := \sum_{k\in\N} S_{k,i}
\]
are converging sums in $L_p^+I$ and $L_\infty^+I$, respectively, i.e., the sequences of partial sums $(\sum\limits_{k\in\N_N} q_{k,i})_{N\in\N}$ and $(\sum\limits_{k\in\N_N} S_{k,i})_{N\in\N}$ converges in $L_p^+I$ and $L_\infty^+I$, respectively.

By Proposition \ref{prop5.1} and Theorem \ref{thm5.2}, we immediately obtain the following results concerning infinite series of fractal functions from $\fL_p^+$.

\begin{proposition}\label{prop5.2}
Under the hypotheses of Theorem \ref{thm5.2}, the infinite series $\sum\limits_{m\in\N} f_m(\bq)$ of fractal functions $f_m(\bq)\in\fL_p^+$ converges in $L_p^+I$ to a fractal function $f(\bq)\in \fL_p^+$ satisfying the self-referential equation
\[
f(\bq) = \sum_{i\in\N_n} q_{i}\circ l_i^{-1} \chi_{I_i} + \sum_{i\in\N_n} (S_{i} \circ l_i^{-1})\cdot (f(\bq)\circ l_i^{-1}) \chi_{I_i},
\]
where $q_{i} := \sum\limits_{k\in\N} q_{k,i}$ and $S_i := \sum\limits_{k\in\N} S_{k,i}$.
\end{proposition}
\subsection{Fractal Interpolation in $C^+I$}
In this section, we construct continuous fractal functions from $C^+I$ into itself. To this end, we refer to Subsection \ref{sect4.2} for the general idea. Under the same assumption on the contractive homeomorphisms $l_i$ but now with $f, Lf, S_i\in C^+I$, we define a semi-affine RB operator $T: C^+I\to (\R_0^+)^{I}$ by
\be
T g := f + (S_i\circ l_i^{-1})\cdot (g-L f)\circ l_i^{-1}, \quad \textrm{on} \,\, I_i,\,\, i\in \N_n,
\ee
where the subintervals $I_i$ are defined as in \eqref{part}. It is clear that on the open intervals $(x_{i-1},x_i)$, $Tg$ is continuous and by imposing the interior join-up conditions \eqref{joinup}, $Tg$ is also continuous at $x_j$, $j\in\N_{n-1}$. Hence, $Tg\in C^+I$. To determine a condition that enforces $T$ to be contractive on $C^+I$, we proceed in the usual manner using the specific expression of the RB operator $T$ and properties of $\max$ and $\min$:
\begin{align*}
d(Tg,Th) & = \max_{x\in I}\left\{\max\left\{(S_i\circ l_i^{-1}(x))(g\circ l_i^{-1})(x), (S_i\circ l_i^{-1}(x))(h\circ l_i^{-1})(x)\right\}\right.\\
& \quad - \left.\min\left\{(S_i\circ l_i^{-1}(x))(g\circ l_i^{-1})(x), (S_i\circ l_i^{-1}(x))(h\circ l_i^{-1})(x)\right\}\right\}\\
& \leq S_\infty\, d(g,h),
\end{align*}
where $S_\infty := \max\limits_{i\in\N_n}\{\n{S_i}_\infty\}$. Therefore, the following theorem holds.
\begin{theorem}
Under the afore-mentioned hypotheses, the semi-affine RB operator $T: C^+I\to C^+I$ defined is contractive on $C^+I$ provided $S_\infty < 1$. The unique fixed point of $T$ is thus a self-referential function $f^*$ satisfying
\be\label{cfix}
f^* := f + (S_i\circ l_i^{-1})\cdot (f^*-L f)\circ l_i^{-1}, \quad \textrm{on} \,\, I_i,\,\, i\in \N_n.
\ee
\end{theorem}
\begin{proof}
The statements follow directly from the above observations.
\end{proof}
\begin{definition}
The self-referential function $f^*$ given by \eqref{cfix} is termed a fractal function of class $C^+I$. The collection of all such functions is denoted by $\fC^+$.
\end{definition}
\begin{remark}
As was noted below Theorem \ref{thm5.1}, the fixed point $f^*\in\fC^+$ depends on the tuple of contractive homeomorphisms $\bl$ and scaling functions $\bS$, as well as the bounded semi-linear operator $L$. Hence, as was the case for fractal interpolation in $L_p^+I$, for fixed $\bl$ and $L$, the fixed point $f^*$ describes an entire family of fractal functions $f^*(\bS)$ for the ``germ function" $f$. Also note that if $\bS = \boldsymbol{0}$, then $f^*(\boldsymbol{0}) = f$.
\end{remark}

One can also consider sequences and infinite series of fractal functions of type $C^+I$. The results obtained in Theorem \ref{thm5.2} and Proposition \ref{prop5.2} carry over to the present setting by simply replacing the $d_p$-metric by the $\sup$-metric. We leave the details to the interested reader.
\subsection{A Fractal Operator for Complete Semi-Vector Spaces}
Let $X^+I \in\{L_p^+I,C^+I\}$ and let $\fX^+\in\{\fL_p^+, \fC^+\}$ be the corresponding space of fractal functions of this type. 

We define a semi-linear fractal operator $\cF^{\bS} : X^+I\to\fX^+$ for the $n$-tuple of scaling functions $\bS := (S_i)_{i\in\N_n}$, as that operator which maps a function $f\in X^+I$ to the self-referential function $f(\bS)$ satisfying the fixed point equation
\be\label{FS}
\cF^{\bS} f(\bS) = f(\bS) = f + (S_i \circ l_i^{-1})\cdot ({f(\bS)}-Lf)\circ l_i^{-1}, \quad \textrm{on} \,\, I_i,\,\, i\in \N_n.
\ee
Here, the contractions $(l_i)$ are defined for the appropriate setting and $L$ is the bounded semi-linear operator introduced above. 

Fractal operators on different function spaces were first introduced in \cite{NC}.
\begin{theorem}\label{thm5.4}
For $p\in [1,\infty)$, $(C^+I, d_p)$ is dense in $(L_p^+I, d_p)$.
\end{theorem}
\begin{proof}
The proof follows readily from \cite[Theorem 3.14]{Ru} restricted to positive simple functions, i.e., to simple functions whose coefficients are from $\R_0^+$. For details about integration of positive functions, see \cite[Chapter 9]{S}.
\end{proof}

Let $\tn{\cdot} \in \{\n{\cdot}, \n{\cdot}_p\}$. We note that as $C^+I$ is dense in $L_p^+I$, the Hahn-Banach theorem implies that there exists a norm-preserving extension of the operator $L: C^+I \to C^+I$ to $L_p^+I \to L_p^+I$. We denote this extension again by $L$. Recall that in the new notation, $S_\infty := \max\limits_{i\in\N_n}\{\tn{S_i}\}$

Then, using the fixed point equation \eqref{FS}, we see that
\[
\tn{f(\bS) - f} \leq \frac{S_\infty}{1-S_\infty} \tn{I-L}\,\tn{f},
\]
and, with $f(\bS) = \cF^\bS f$ and use of the triangle inequaltiy, this becomes
\[
\tn{\cF^\bS f} \leq \left(1+ \frac{S_\infty}{1-S_\infty} \tn{I-L}\right)\,\tn{f}.
\]
Therefore, we arrive at the following result.
\begin{proposition}\label{prop5.3}
Suppose that the constant function $1$ is contained in the point spectrum of the semi-linear fractal operator $\cF^\bS: X^+\to \fX^+$ as defined in \eqref{FS}. Then $\cF^\bS$ satisfies the following $\tn{\cdot}$-estimates:
\be
1\leq \tn{\cF^\bS} \leq \left(1+ \frac{S_\infty}{1-S_\infty} \tn{I-L}\right)
\ee
\end{proposition}
\begin{proof}
Only the left inequality needs to be shown. This, however, follows immediately from the fact that the constant function $1$ is contained in the point spectrum of $\cF^\bS$.
\end{proof}
Note that in Proposition \ref{prop5.3} the right-hand estimate always holds, regardless of whether $f=1$ is contained in the point spectrum of $\cF^\bS$ or not.
\section*{Conclusion and Future Research Directions}
This paper introduced the novel concept of fractal interpolation in complete semi-vector spaces. In particular, the existence of fractal functions satisfying the fractal interpolation problem in the complete semi-vector spaces $C^+I$ and $L_p^+I$ was investigated. Sequences and series of such fractal functions were briefly considered as well. Moreover, a semi-linear fractal operator was defined and some of properties presented.

The following future research directions present themselves:
\begin{itemize}
\item Instead of using Banach contractions, extend the approach to, for instance,  Rakotch, Matkowsi, $\varphi$- or $\mathsf{F}$-contractions. For definitions of these more general contractions which all imply the existence of a unique fixed point, we refer to \cite{BW,Mat,Rak,RZ,W}. 
\item Consider other function spaces such as H\"older, Sobolev-Slobodeckij, Besov, and Triebel-Lizorkin, to name a few.
\end{itemize}

\end{document}